\documentclass[a4paper, 12pt]{amsproc}
\allowdisplaybreaks[1]

\usepackage{amsmath, amsthm, amssymb, mathtools, mathrsfs, stmaryrd}
\usepackage{ascmac}
\usepackage{comment}
\usepackage{bm}
\usepackage{ytableau}

\usepackage{hyperref}

\usepackage{tikz}

\usepackage{graphicx}
\usepackage{here}
\usepackage{time}
\usepackage[abbrev]{amsrefs}

\usepackage{xcolor}
\usepackage[capitalize,nameinlink,noabbrev,nosort]{cleveref}
\hypersetup{
 colorlinks=true,       
 linkcolor=blue,          
 citecolor=blue,        
 filecolor=blue,      
 urlcolor=blue,           
}

\usepackage{fullpage}

\makeatletter
\@namedef{subjclassname@2020}{%
  \textup{2020} Mathematics Subject Classification}
\makeatother

\newtheorem{theoremcounter}{Theorem Counter}[section]

\theoremstyle{definition}
\newtheorem{definition}[theoremcounter]{Definition}
\newtheorem{remark}[theoremcounter]{Remark}
\newtheorem*{comparisonremark}{Remark}
\newtheorem{example}[theoremcounter]{Example}

\theoremstyle{plain}
\newtheorem{lemma}[theoremcounter]{Lemma}

\newtheorem{corollary}[theoremcounter]{Corollary}

\newtheorem{theorem}[theoremcounter]{Theorem}

\numberwithin{equation}{section}

\newcommand{\Res}{\mathrm{Res}}

\DeclareMathOperator{\ReNew}{Re}
\renewcommand{\Re}{\ReNew}

\begin{document}
\author{Takashi Miyagawa}
\address[Takashi Miyagawa]{Onomichi City University,  1600-2 Hisayamada-cho, Onomichi, Hiroshima, 722-8506, Japan} 
\email{miyagawa@onomichi-u.ac.jp}



\begin{abstract}
We study the Laurent coefficients of the Barnes multiple zeta function
with complex parameters in a common open half-plane and a fixed
holomorphic determination of the logarithm.
At the highest pole, we derive explicit limit formulae for every
regular Laurent coefficient in terms of finite multiple sums and
logarithmic correction terms.
At the lower possible poles, we establish all-order relations with
the Taylor coefficients at the origin and their parameter derivatives.
We also determine the large-order behavior by subtracting all principal
parts. The remaining function is entire, so Cauchy's estimate separates
explicit residue contributions from a remainder that decays faster
than any fixed geometric rate.
The neighboring residues yield the limiting even and odd subsequences,
with vanishing residues and cancellations accounted for.
The constant-term cases recover known finite-part representations;
the main focus is their extension to higher Laurent coefficients.
\end{abstract}


\title{Laurent series expansions for the Barnes multiple zeta function}

\maketitle

\section{Introduction}

Let $r$ be a positive integer and let $s=\sigma+it$
($\sigma,t\in\mathbb{R}$)
be a complex variable.
For a fixed real number $-\pi<\phi<\pi$, let
\[
\mathcal{H}_{\phi}
=
\left\{
\rho e^{i\theta}\in\mathbb{C}
\;\middle|\;
\rho>0,\ 
\phi-\frac{\pi}{2}<\theta<\phi+\frac{\pi}{2}
\right\}.
\]
This is an open half-plane bounded by a line through the origin.
Throughout the paper, the Barnes parameters are assumed to satisfy
$a,w_1,\ldots,w_r\in\mathcal{H}_{\phi}$.
We keep $\phi$ fixed when varying the parameters locally.
We use a fixed, but otherwise arbitrary, holomorphic logarithm $\log$
on $\mathcal{H}_{\phi}$ and set $z^s=\exp(s\log z)$ for
$z\in\mathcal{H}_{\phi}$.
All complex powers, logarithms, Laurent coefficients, and parameter
derivatives below use this determination, except where the classical
Hurwitz comparison explicitly specifies the real logarithm.
The defining angle condition implies $\Re(e^{-i\phi}z)>0$
for $z\in\mathcal{H}_{\phi}$.
For $x_1,\ldots,x_r\geq0$, put
\[
\alpha=\Re\left(e^{-i\phi}a\right)>0,
\qquad
\beta_j=\Re\left(e^{-i\phi}w_j\right)>0.
\]
Then
\[
\begin{aligned}
\left|a+\sum_{j=1}^{r}w_jx_j\right|
&\geq
\Re\left(e^{-i\phi}\left(a+\sum_{j=1}^{r}w_jx_j\right)\right)
=\alpha+\sum_{j=1}^{r}\beta_jx_j.
\end{aligned}
\]
In particular, $a+\sum_{j=1}^{r}w_jx_j\in\mathcal{H}_{\phi}$.
Taking $C=\min\{\alpha,\beta_1,\ldots,\beta_r\}>0$, we obtain
\[
\left|a+\sum_{j=1}^{r}w_jx_j\right|
\geq\alpha+\sum_{j=1}^{r}\beta_jx_j
\geq C\left(1+\sum_{j=1}^{r}x_j\right).
\]
We shall use this lower bound throughout; $C$ may be chosen uniformly
on compact subsets of $\mathcal{H}_{\phi}^{r+1}$.

Under this common half-plane condition, the Barnes multiple zeta function,
introduced by Barnes \cites{Barnes1899,Barnes1901,Barnes1904},
is defined by
\begin{align}\label{zeta_r}
\zeta_r(s,a;\boldsymbol{w})
&=
\sum_{m_1=0}^{\infty}\cdots
\sum_{m_r=0}^{\infty}
\frac{1}
     {(a+m_1w_1+\cdots+m_rw_r)^s}\\
&=
\sum_{m_1=0}^{\infty}\cdots
\sum_{m_r=0}^{\infty}
\frac{1}
     {(a+\boldsymbol{m}\cdot\boldsymbol{w})^s} \nonumber
\end{align}
where
$\boldsymbol{m}=(m_1,\ldots,m_r), \boldsymbol{w}=(w_1,\ldots,w_r)$.

This function is a natural generalization of the Hurwitz zeta function
\begin{align}\label{zeta_H_intro}
\zeta_H(s,a)
=
\sum_{m=0}^{\infty}
\frac{1}{(m+a)^s}
\quad (\sigma>1).
\end{align}
Here the Hurwitz case has $w_1=1$, so $a$ and $1$ must lie in a
common half-plane as above.
It is well known that the series \eqref{zeta_r}
converges absolutely for $\sigma>r$ and admits a meromorphic continuation
to the whole complex plane; see also \cite{Noronha2017}.
Indeed, for $s$ in a compact set $K$ and $z\in\mathcal{H}_{\phi}$,
\[
|z^{-s}|\leq A_K|z|^{-\sigma},
\]
where $A_K>0$ is independent of $z$.
Together with the preceding lower bound, this gives locally uniform
absolute convergence for $\sigma>r$, also locally in the parameters.
Moreover, its only possible singularities are simple poles at
$s=1,2,\ldots,r$.
Also, the residues at these poles are explicitly described by the Barnes
Bernoulli polynomials.
More precisely, let
$B_n^{(r)}(x;\boldsymbol{w})$
for $x\in\mathbb{C}$ and $w_1,\ldots,w_r\in\mathcal{H}_{\phi}$
be defined by the generating function, valid for $t$ sufficiently close to $0$,
\begin{align}\label{Bernoulli_intro}
\frac{t^r e^{xt}}
     {(e^{w_1t}-1)\cdots(e^{w_rt}-1)}
=
\sum_{n=0}^{\infty}
B_n^{(r)}(x;\boldsymbol{w})
\frac{t^n}{n!}.
\end{align}
The singularity at $t=0$ is removable since each $w_j\neq0$.
Taking $x=a\in\mathcal{H}_{\phi}$, the residue formula of
Barnes \cite{Barnes1904} reads
\begin{align}\label{residue_intro}
\Res_{s=j}
\zeta_r(s,a;\boldsymbol{w})
=
\frac{(-1)^{r-j}}
     {(j-1)!(r-j)!}
B_{r-j}^{(r)}(a;\boldsymbol{w})
\qquad
(1\le j\le r).
\end{align}
If the right-hand side vanishes, the corresponding singularity is removable.
The finite parts, that is, the constant terms of the Laurent expansions,
have also been studied. Barnes' relation with the multiple polygamma
functions is recalled by Noronha \cite{Noronha2017}*{Eq.~(6)}.
Noronha gives series and limit representations for the finite parts
at every possible pole in arbitrary dimension under the common
half-plane condition \cite{Noronha2017}*{Eqs.~(12) and~(18)}.
He also gives integral representations under the additional conditions
$\Re a>0$ and $\Re w_i>0$ for his real-axis integrals
\cite{Noronha2017}*{Section~3, Eq.~(26)}.
These results provide the finite-part cases with which we compare
our formulae below.

The higher Laurent coefficients $k\geq1$ are the main subject of the
present paper. Our purpose is to give explicit and unified formulae
valid at every order, retaining the known finite parts as the
order-zero cases.
Higher coefficients can also be extracted by further expansion of
existing analytic continuation formulae, including Barnes' series
reproduced in \cite{Noronha2017}*{Theorem~2.1 and Proposition~2.2};
related integral and series representations are given in
\cites{Ruijsenaars2000,Spreafico2009}.
Here we make the order-dependent logarithmic correction terms explicit
at the highest pole and prove the resulting limit formulae by
Euler--Maclaurin summation. At the lower possible poles, we give an
all-order relation with the Taylor coefficients at $s=0$.
A third result describes the large-order behavior through the residues
at the other poles, independently of these coefficient representations.

The Barnes multiple zeta function is closely related to the Barnes multiple gamma function
$\Gamma_r(x;\boldsymbol{w})$, for which we use the zeta-regularized
normalization specified by the following identity for
$x\in\mathcal{H}_{\phi}$:
\[
\log \Gamma_r(x;\boldsymbol{w})
:=
\left.
\frac{\partial}{\partial s}
\zeta_r(s,x;\boldsymbol{w})
\right|_{s=0}.
\]
Its logarithmic derivatives
\[
\psi_r^{(k)}(x;\boldsymbol{w})
:=
\frac{d^k}{dx^k}
\log \Gamma_r(x;\boldsymbol{w})
\qquad (k\ge 1)
\]
are called the Barnes multiple polygamma functions; the derivatives
are holomorphic derivatives with respect to the complex variable $x$.
These functions naturally arise in the study of special values and Laurent expansions of
$\zeta_r(s,x;\boldsymbol{w})$.

We now state our results for the Laurent coefficients of the Barnes
multiple zeta function \eqref{zeta_r}.
Write its local expansions at $s=r$ and $s=j$ ($1\leq j\leq r-1$) as
    \begin{align}
    & \zeta_r(s,a;\boldsymbol{w}) = \frac{1}{(r-1)!\,w_1\cdots w_r}\frac{1}{s-r}   + \gamma_0(r,a;\boldsymbol{w}) + \sum_{k=1}^\infty \gamma_k(r,a;\boldsymbol{w})(s-r)^k,\label{Laurent_s=r}   \\
    & \zeta_r(s,a;\boldsymbol{w}) = \frac{(-1)^{r-j}B_{r-j}^{(r)}(a;\boldsymbol{w})}{(j-1)!(r-j)!}\frac{1}{s-j}   + \gamma_0(j,a;\boldsymbol{w}) + \sum_{k=1}^\infty \gamma_k(j,a;\boldsymbol{w})(s-j)^k. \label{Laurent_s=j}
    \end{align}
With this notation, the finite part is
\[
\gamma_0(j,a;\boldsymbol{w})
=\operatorname{F.P.}_{s=j}\zeta_r(s,a;\boldsymbol{w}),
\qquad 1\leq j\leq r.
\]
Here we adopt the notation $\operatorname{F.P.}$ used by Noronha
\cite{Noronha2017} for the ``finite part'', namely the constant term
in the Laurent expansion at the indicated point.
Theorem~\ref{th:Main_Theorem1} treats all $k\geq0$ at $s=r$;
its $k=0$ case recovers Noronha's highest-pole finite-part limit.
Theorem~\ref{th:Main_Theorem2} treats the lower poles in all orders,
with the classical Barnes finite-part relation as its $k=0$ case.
The higher coefficients $k\geq1$ are the main subject of the present
all-order formulation. We also write
$\gamma_{-1}(j,a;\boldsymbol{w})=\Res_{s=j}\zeta_r(s,a;\boldsymbol{w})$
for $1\leq j\leq r$, including the case of a zero residue.

\medskip

\begin{definition}\label{def:cqr}
For each integer $r\ge2$, let the constants
$c_q^{(r)}\ (q=0,1,2,\ldots)$
be defined by the generating function
\[
\frac{1}
{(1+\varepsilon)(2+\varepsilon)\cdots(r-1+\varepsilon)}
=
\sum_{q=0}^{\infty}
c_q^{(r)}\varepsilon^q,
\qquad |\varepsilon|<1.
\]
\end{definition}

\medskip

\begin{theorem}\label{th:Main_Theorem1}
Let $r\geq2$ and $\boldsymbol{w}=(w_1,\ldots,w_r)$, and assume
$a,w_1,\ldots,w_r\in\mathcal{H}_{\phi}$ for some $-\pi<\phi<\pi$.
Let the constants $c_q^{(r)}$ be defined by
Definition \ref{def:cqr}.
For each $k=0,1,2,\ldots$, the Laurent coefficient
$\gamma_k(r,a;\boldsymbol{w})$ is given by
\begin{align}
\gamma_k(r,a;\boldsymbol{w})
&=
\lim_{M\to\infty}
\Bigg\{
\frac{(-1)^k}{k!}
\sum_{m_1=0}^M\cdots\sum_{m_r=0}^M
\frac{
\log^k(a+\boldsymbol{m}\cdot\boldsymbol{w})
}{
(a+\boldsymbol{m}\cdot\boldsymbol{w})^r
}
\nonumber\\
&\qquad
-
\frac{1}{w_1\cdots w_r}
\sum_{q=0}^{k+1}
c_q^{(r)}
\frac{(-1)^{k+1-q}}{(k+1-q)!}
\sum_{\emptyset\neq I\subseteq\{1,\ldots,r\}}
(-1)^{|I|}
\nonumber\\
&\hspace{55mm}\times
\log^{k+1-q}
\left(
a+\sum_{i\in I}w_iM
\right)
\Bigg\}.
\label{eq:gamma-r-main}
\end{align}
Here and below, limits of finite sums are taken over positive integers $M$.
In particular,
\begin{align}
\gamma_0(r,a;\boldsymbol{w})
&=
\lim_{M\to\infty}
\Bigg\{
\sum_{m_1=0}^M\cdots\sum_{m_r=0}^M
\frac{1}{(a+\boldsymbol{m}\cdot\boldsymbol{w})^r}
\nonumber\\
&\qquad
+
\frac{1}{(r-1)!w_1\cdots w_r}
\sum_{\emptyset\neq I\subseteq\{1,\ldots,r\}}
(-1)^{|I|}
\log
\left(
a+\sum_{i\in I}w_iM
\right)
\nonumber\\
&\qquad
-
\frac{H_{r-1}}
{(r-1)!w_1\cdots w_r}
\Bigg\},
\label{eq:gamma0-r-main}
\end{align}
where
\[
H_{r-1}
=
1+\frac12+\cdots+\frac1{r-1}.
\]
\end{theorem}

\begin{comparisonremark}
For $k=0$, \eqref{eq:gamma0-r-main} is equivalent to
Noronha's finite-part limit \cite{Noronha2017}*{Eq.~(18)} with
his dimension $d$ and pole index $q$ both equal to $r$.
After translating his notation, the logarithmic correction is the
nonempty-subset sum in \eqref{eq:gamma0-r-main}, and the constant
correction is $-H_{r-1}/((r-1)!w_1\cdots w_r)$.
The only remaining difference is the summation cube:
his formula uses $\{0,\ldots,M-1\}^r$ instead of
$\{0,\ldots,M\}^r$. The common half-plane lower bound gives
\[
\left|
\sum_{\substack{0\leq m_1,\ldots,m_r\leq M\\
                 \max_i m_i=M}}
(a+\boldsymbol{m}\cdot\boldsymbol{w})^{-r}
\right|
\leq \bigl((M+1)^r-M^r\bigr)(CM)^{-r}
=O(M^{-1}),
\]
so the limits coincide, with the same logarithm determination.
The cases $k\geq1$ of Theorem~\ref{th:Main_Theorem1} explicitly
extend this type of limit formula beyond the finite part.
\end{comparisonremark}

\medskip


\medskip
\begin{theorem}\label{th:Main_Theorem2}
Let $r\geq2$ and assume $a,w_1,\ldots,w_r\in\mathcal{H}_{\phi}$
for some $-\pi<\phi<\pi$.
For each $j=1,\ldots,r-1$, let the constants $d_{j,n}$
$(n=-1,0,1,\ldots)$ be defined by
\[
\frac{1}{(s)_j}
=
\sum_{n=-1}^{\infty} d_{j,n}s^n,
\qquad
(s)_j=s(s+1)\cdots(s+j-1),
\]
where the expansion is valid for $0<|s|<1$.
Then, for $j=1,\ldots,r-1$ and
$k=-1,0,1,2,\ldots$, the Laurent coefficients of
\eqref{Laurent_s=j} satisfy
\[
\gamma_k(j,a;\boldsymbol{w})
=
(-1)^j
\sum_{m=0}^{k+1}
d_{j,k-m}
\frac{1}{m!}
\frac{\partial^j}{\partial a^j}
\zeta_r^{(m)}(0,a;\boldsymbol{w}).
\]
Here $\zeta_r^{(m)}$ denotes differentiation with respect to $s$,
and $\partial^j/\partial a^j$ denotes holomorphic differentiation in
$a\in\mathcal{H}_{\phi}$ with $\boldsymbol{w}$ fixed.
\end{theorem}

\medskip

\begin{comparisonremark}
Let $H_0=0$. Since
\[
\frac{1}{(s)_j}
=\frac{1}{(j-1)!}\left(\frac1s-H_{j-1}+O(s)\right),
\]
we have $d_{j,-1}=1/(j-1)!$ and
$d_{j,0}=-H_{j-1}/(j-1)!$.
The case $k=-1$ of Theorem~\ref{th:Main_Theorem2} gives
\[
\gamma_{-1}(j,a;\boldsymbol{w})
=\frac{(-1)^j}{(j-1)!}
\frac{\partial^j}{\partial a^j}\zeta_r(0,a;\boldsymbol{w}).
\]
Substitution in its $k=0$ case yields
\begin{align*}
\operatorname{F.P.}_{s=j}\zeta_r(s,a;\boldsymbol{w})
&=\frac{(-1)^j}{(j-1)!}
\frac{\partial^j}{\partial a^j}\zeta_r'(0,a;\boldsymbol{w})\\
&\quad-H_{j-1}\Res_{s=j}\zeta_r(s,a;\boldsymbol{w}).
\end{align*}
This is the classical Barnes relation recalled by Noronha
\cite{Noronha2017}*{Eq.~(6)}.
The normalization factor $\rho_B(\boldsymbol{w})$ in his definition
of the multiple gamma function is independent of $a$, so it disappears
under these derivatives.
The cases $k\geq1$ extend this finite-part relation to all orders
by expanding the same parameter-differentiation identity.
\end{comparisonremark}

\medskip

Our third main result, Theorem~\ref{th:all-poles-coefficients} in
Section~\ref{sec:large-order}, concerns $k\to\infty$ with the parameters
fixed. Subtracting all principal parts leaves an entire function.
Cauchy's estimate then gives an exact sum of contributions from the
other poles with a remainder of $O(R^{-k})$ for every fixed $R>0$.
Keeping only the neighboring contributions gives an $O(2^{-k})$
remainder and, at intermediate poles, the limits
\[
\lim_{n\to\infty}\gamma_{2n}(j,a;\boldsymbol w)=R_{j-1}-R_{j+1},
\qquad
\lim_{n\to\infty}\gamma_{2n+1}(j,a;\boldsymbol w)=-R_{j-1}-R_{j+1},
\]
where $R_q=\Res_{s=q}\zeta_r(s,a;\boldsymbol w)$ and
$2\leq j\leq r-1$. A zero residue removes the corresponding
contribution, and equidistant contributions may cancel on one parity.
The all-poles formula extends the double-zeta asymptotic forms in
\cite{Miyagawa2025_LEBDZ}*{Theorems~2.6 and~2.7} to arbitrary dimension.
The proof uses only the meromorphic continuation and Cauchy's formula;
it does not require parameter asymptotics or estimates from the first
two theorems.

For comparison, in the classical Hurwitz case with $a>0$ and the real
logarithm on the positive real axis, the coefficients in the Laurent
expansion at $s=1$ are normalized generalized Euler--Stieltjes constants.

The Hurwitz zeta function admits the expansion
\[
\zeta_H(s,a)=\frac{1}{s-1}+\gamma_0(a)+\sum_{k=1}^\infty \gamma_k(a)(s-1)^k,
\]
where
\[
\gamma_k(a)=\frac{(-1)^k}{k!} \lim_{M\to\infty}
\left(
\sum_{m=0}^M \frac{\log^k(m+a)}{m+a}
-
\frac{\log^{k+1}(M+a)}{k+1}
\right).
\]
In this comparison, $\gamma_k:=\gamma_k(1)$ and $\gamma_k(a)$ denote
the coefficients of the corresponding powers in the Laurent expansions.
Thus they are $(-1)^k/k!$ times the classical Stieltjes constants.
With this normalization, the estimate of Zhang and Williams
\cite{ZhangWilliams1994} for the classical Euler--Stieltjes constants is equivalent to
\begin{equation}
    |\gamma_k|
    \le
    \frac{3+(-1)^k}{(2\pi)^k}
    \cdot
    \frac{(2k)!}{k!\,k^{k+1}}
    \qquad (k\ge 1).
    \label{estim_gamma_k}
\end{equation}
For a discussion of equivalent estimates for the Stieltjes constants,
see also Finch \cite{Finch2003}*{\S 2.21}.
Similarly, rewriting the estimate of Berndt \cite{Berndt1972} according to the normalization adopted here, we have
\begin{equation}
    \left|
        \gamma_k(a)
        -
        \frac{(-1)^k(\log a)^k}{a k!}
    \right|
    \le
    \frac{3+(-1)^k}{k\pi^k}
    \qquad
    (0<a\le 1,\ k\ge 1).
    \label{estim_gamma_k(a)}
\end{equation}
These bounds concern the coefficient order, rather than asymptotics
in the parameters of the zeta function.

As related work, Matsumoto, Onozuka, and Wakabayashi \cite{MatsumotoOnozukaWakabayashi2020} studied multivariable Laurent-type expansions of the Euler--Zagier multiple zeta function
\[
\zeta_r(s_1, \dots, s_r)
=
\sum_{m_1=1}^\infty \cdots \sum_{m_r=1}^\infty
\frac{1}{
m_1^{s_1}
(m_1+m_2)^{s_2}
\cdots
(m_1+\cdots+m_r)^{s_r}
},
\]
and introduced analogues of Euler--Stieltjes constants in that setting.
This multivariable problem is distinct from the single-variable
Barnes expansions considered here.

The Laurent expansions in the double-zeta case were studied in
\cite{Miyagawa2025_LEBDZ}. We now give the $r=2$ specializations of
Theorems~\ref{th:Main_Theorem1} and~\ref{th:Main_Theorem2}, with
$a,w_1,w_2\in\mathcal{H}_{\phi}$ and the coefficient normalization used here.

\begin{example}\label{ex:double_zeta1}
For $r=2$, let $a,w_1,w_2\in\mathcal{H}_{\phi}$.
Since
\[
c_q^{(2)}=(-1)^q
\qquad (q=0,1,2,\ldots),
\]
Theorem \ref{th:Main_Theorem1} gives, for
$k=0,1,2,\ldots$,
\begin{align}
\gamma_k(2,a;w_1,w_2)
&=
\lim_{M\to\infty}
\frac{(-1)^k}{k!}
\Bigg\{
\sum_{m_1=0}^M\sum_{m_2=0}^M
\frac{
\log^k(a+m_1w_1+m_2w_2)
}{
(a+m_1w_1+m_2w_2)^2
}
\nonumber\\
&\qquad
-
\frac{k!}{w_1w_2}
\Bigg[
1+
\sum_{j=1}^{k+1}
\frac{1}{j!}
\Bigl\{
\log^j(a+w_1M)
+
\log^j(a+w_2M)
\nonumber\\
&\hspace{65mm}
-
\log^j(a+w_1M+w_2M)
\Bigr\}
\Bigg]
\Bigg\}.
\label{Euler_Stieltjes_2}
\end{align}
In particular, putting $k=0$ in
\eqref{Euler_Stieltjes_2}, we obtain
\begin{align}
\gamma_0(2,a;w_1,w_2)
&=
\lim_{M\to\infty}
\Bigg\{
\sum_{m_1=0}^M\sum_{m_2=0}^M
\frac{1}{(a+m_1w_1+m_2w_2)^2}
\nonumber\\
&\qquad
-
\frac{1}{w_1w_2}
\Bigl\{
1+\log(a+w_1M)+\log(a+w_2M)
\nonumber\\
&\qquad\qquad\qquad
-\log(a+w_1M+w_2M)
\Bigr\}
\Bigg\}.
\label{gamma_0}
\end{align}
\end{example}

\begin{example}
Taking $j=1$ in Theorem~\ref{th:Main_Theorem2}, we have
$1/(s)_1=1/s$, and hence the following relation between the Taylor
coefficients at $s=0$ and the Laurent coefficients at $s=1$:
\begin{align}
\gamma_k(1,a;w_1,w_2)
=
-\frac{1}{(k+1)!}
\frac{\partial}{\partial a}
\zeta_2^{(k+1)}(0,a;w_1,w_2),
\label{th:Main_Theorem2-1}
\end{align}
for $k=-1,0,1,2,\ldots$.
In particular, since
\[
\zeta_2'(0,a;w_1,w_2)
=
\log\Gamma_2(a;w_1,w_2),
\]
we have
\[
\gamma_0(1,a;w_1,w_2)
=
-\frac{\partial}{\partial a}
\log\Gamma_2(a;w_1,w_2).
\]
Furthermore, expanding
$\zeta_2(2+\varepsilon,a;w_1,w_2)
=-(1+\varepsilon)^{-1}\partial_a\zeta_2(1+\varepsilon,a;w_1,w_2)$
at $\varepsilon=0$ gives
\begin{align}
\sum_{l=-1}^k
(-1)^{k-l+1}
\frac{\partial}{\partial a}
\gamma_l(1,a;w_1,w_2)
=
\gamma_k(2,a;w_1,w_2),
\label{th:Main_Theorem2-2}
\end{align}
for $k=-1,0,1,2,\ldots$.
\end{example}

\medskip
\begin{example}
For $k=0,1$, formula \eqref{th:Main_Theorem2-1} gives
\begin{align*}
\gamma_0(1,a;w_1,w_2)
&=-\frac{\partial}{\partial a}\zeta_2'(0,a;w_1,w_2)
=-\psi_2^{(1)}(a;w_1,w_2),\\
\gamma_1(1,a;w_1,w_2)
&=-\frac12\frac{\partial}{\partial a}\zeta_2''(0,a;w_1,w_2).
\end{align*}
Here the primes on $\zeta_2$ denote derivatives with respect to $s$.
They must not be confused with derivatives with respect to $a$
of $\log\Gamma_2$, which define the multiple polygamma functions.
\end{example}
\medskip

\section{Auxiliary lemmas and proofs}
\begin{lemma}\label{lem:multiple_integral}
Let $r\geq1$ be an integer, $M>0$ be real, and
$\boldsymbol{w}=(w_1,\ldots,w_r)$, and assume
$a,w_1,\ldots,w_r\in\mathcal{H}_{\phi}$ for some $-\pi<\phi<\pi$.
The identities below hold for any single holomorphic determination of the
logarithm on $\mathcal{H}_{\phi}$ used consistently in all the powers.
Put
\[
J_{r,M}(s,a)
=
\int_0^M\cdots\int_0^M
(a+w_1x_1+\cdots+w_rx_r)^{-s}
\,dx_1\cdots dx_r.
\]
Then, for $s\notin\{1,2,\ldots,r\}$,
\[
J_{r,M}(s,a)
=
\frac{1}{w_1\cdots w_r}
\frac{1}{(s-1)(s-2)\cdots(s-r)}
\sum_{I\subseteq\{1,\ldots,r\}}
(-1)^{|I|}
\left(
a+\sum_{i\in I}w_iM
\right)^{r-s}.
\]
The apparent singularities on the right-hand side at
$s=1,\ldots,r$ are removable, and the identity extends to all
$s\in\mathbb{C}$.
Moreover, for $\sigma>r$, the absolutely convergent integral
\[
J_r(s,a)
:=
\int_0^\infty\cdots\int_0^\infty
(a+w_1x_1+\cdots+w_rx_r)^{-s}
\,dx_1\cdots dx_r
\]
is given by
\[
J_r(s,a)
=
\frac{a^{r-s}}
{w_1\cdots w_r
(s-1)(s-2)\cdots(s-r)}.
\]
\end{lemma}

\begin{proof}
We first suppose $s\notin\{1,\ldots,r\}$ and prove the formula for
$J_{r,M}(s,a)$ by induction on $r$.
The variables $x_1,\ldots,x_r$ remain real and nonnegative.
The common half-plane condition ensures that every integration argument,
including the shifted parameter $a+w_rM$, lies in $\mathcal{H}_{\phi}$.
For $r=1$,
\[
J_{1,M}(s,a)
=
\int_0^M(a+w_1x)^{-s}\,dx
=
\frac{a^{1-s}-(a+w_1M)^{1-s}}
{w_1(s-1)},
\]
which is the desired formula.
Suppose that the assertion holds for $r-1$.
Put
\[
A=a+w_1x_1+\cdots+w_{r-1}x_{r-1}.
\]
Then
\begin{align*}
J_{r,M}(s,a)
&=
\int_0^M\cdots\int_0^M
\left\{
\int_0^M(A+w_rx_r)^{-s}\,dx_r
\right\}
dx_1\cdots dx_{r-1}
\\
&=
\frac{1}{w_r(1-s)}
\left\{
J_{r-1,M}(s-1,a+w_rM)
-
J_{r-1,M}(s-1,a)
\right\}.
\end{align*}
Substituting the induction hypothesis into the preceding identity and using
$
1/(1-s)=-1/(s-1),
$
we obtain
\begin{align*}
J_{r,M}(s,a)
&=
\frac{1}{w_1\cdots w_r}
\frac{1}{(s-1)\cdots(s-r)}
\Bigg\{
\sum_{J\subseteq\{1,\ldots,r-1\}}
(-1)^{|J|}
\left(
a+\sum_{j\in J}w_jM
\right)^{r-s}
\\
&\qquad\qquad
-
\sum_{J\subseteq\{1,\ldots,r-1\}}
(-1)^{|J|}
\left(
a+w_rM+\sum_{j\in J}w_jM
\right)^{r-s}
\Bigg\}.
\end{align*}
The first sum corresponds precisely to the subsets
$I\subseteq\{1,\ldots,r\}$ not containing $r$:
putting $I=J$, we have
$
(-1)^{|J|}=(-1)^{|I|}.
$
The second sum corresponds to the subsets containing $r$.
Indeed, putting
$
I=J\cup\{r\},
$
we have $|I|=|J|+1$, and therefore
$
-(-1)^{|J|}
=
(-1)^{|J|+1}
=
(-1)^{|I|}.
$
Consequently, the two sums combine to give
\[
\sum_{I\subseteq\{1,\ldots,r\}}
(-1)^{|I|}
\left(
a+\sum_{i\in I}w_iM
\right)^{r-s}.
\]
Equivalently, the convention with denominators $(1-s),\ldots,(r-s)$
and signs $(-1)^{r-|I|}$ gives the same expression, since
\[
\frac{(-1)^{r-|I|}}{(1-s)(2-s)\cdots(r-s)}
=
\frac{(-1)^{|I|}}{(s-1)(s-2)\cdots(s-r)}.
\]
Hence
\[
J_{r,M}(s,a)
=
\frac{1}{w_1\cdots w_r}
\frac{1}{(s-1)\cdots(s-r)}
\sum_{I\subseteq\{1,\ldots,r\}}
(-1)^{|I|}
\left(
a+\sum_{i\in I}w_iM
\right)^{r-s}.
\]
On a finite integration cube the arguments lie in a compact subset of
$\mathcal{H}_{\phi}$. The integrand is entire in $s$ and uniformly
bounded when $s$ ranges over a compact set. Integration therefore
preserves holomorphy, so $J_{r,M}(s,a)$ is entire. The apparent
singularities on the right-hand side at $s=1,\ldots,r$ are removable.
Finally, for a compact set $K\subset\{s:\sigma>r\}$, put
$\sigma_K=\inf_{s\in K}\sigma>r$.
The lower bound from the introduction gives
\[
\left|(a+w_1x_1+\cdots+w_rx_r)^{-s}\right|
\leq A_K(1+x_1+\cdots+x_r)^{-\sigma_K}
\qquad(s\in K).
\]
The right-hand side is integrable on $[0,\infty)^r$.
Dominated convergence therefore implies that the finite integrals
$J_{r,M}(s,a)$ converge to $J_r(s,a)$ locally uniformly for $\sigma>r$.
Absolute convergence also permits successive integration by Fubini's theorem.
For every nonempty $I\subseteq\{1,\ldots,r\}$, the same lower bound gives
\[
\left|\left(a+\sum_{i\in I}w_iM\right)^{r-s}\right|
\leq B_K(1+M)^{r-\sigma_K}
\qquad(s\in K),
\]
where $B_K$ is independent of $M$ and $I$.
Since $r-\sigma_K<0$, these terms tend to zero uniformly on $K$.
Only the term $I=\emptyset$ survives in the finite-integral formula,
so we obtain
\begin{align*}
& \int_0^\infty\cdots\int_0^\infty
(a+w_1x_1+\cdots+w_rx_r)^{-s}
 \,dx_1\cdots dx_r \\
& \qquad =
\frac{a^{r-s}}
{w_1\cdots w_r
(s-1)(s-2)\cdots(s-r)}.
\end{align*}
\end{proof}

\begin{lemma}\label{lem:regularized_difference}
Assume $a,w_1,\ldots,w_r\in\mathcal{H}_{\phi}$ for some $-\pi<\phi<\pi$,
and let $M$ be a positive integer.
Let
\[
D_{r,M}(s)
=
\sum_{m_1=0}^M\cdots\sum_{m_r=0}^M
(a+\boldsymbol{m}\cdot\boldsymbol{w})^{-s}
-
J_{r,M}(s,a).
\]
Then there exists a neighborhood $U$ of $s=r$ such that
$D_{r,M}(s)$ converges locally uniformly on $U$, as $M\to\infty$,
to a function $H_r(s,a;\boldsymbol{w})$ holomorphic on $U$.

For $s\in U$ with $\sigma>r$,
\[
H_r(s,a;\boldsymbol{w})
=
\zeta_r(s,a;\boldsymbol{w})
-
J_r(s,a).
\]
Consequently,
\[
\zeta_r(s,a;\boldsymbol{w})
=
H_r(s,a;\boldsymbol{w})
+
J_r(s,a)
\]
gives the meromorphic continuation of $\zeta_r$ to a neighborhood
of $s=r$.
\end{lemma}

\begin{proof}
For a continuously differentiable function $g$ and an integer $M\geq1$,
the first-order Euler--Maclaurin formula is
\[
\sum_{m=0}^{M}g(m)
=\int_0^M g(x)\,dx+\frac{g(0)+g(M)}{2}
+\int_0^M\widetilde B_1(x)g'(x)\,dx,
\]
where $\widetilde B_1(x)=x-\lfloor x\rfloor-\tfrac12$ is bounded.
Apply this formula successively with
respect to the real variables $x_1,\ldots,x_r$ to the complex-valued function
\[
f_s(\boldsymbol{x})=(a+w_1x_1+\cdots+w_rx_r)^{-s}.
\]
The formula applies to its real and imaginary parts separately.
For a multi-index $\boldsymbol{\mu}=(\mu_1,\ldots,\mu_r)$ with
$\nu=\mu_1+\cdots+\mu_r$, differentiation gives
\[
\partial_{\boldsymbol{x}}^{\boldsymbol{\mu}}f_s(\boldsymbol{x})
=(-1)^\nu(s)_\nu\prod_{j=1}^r w_j^{\mu_j}
(a+\boldsymbol{x}\cdot\boldsymbol{w})^{-s-\nu},
\qquad (s)_0=1.
\]
Choose $U$ to be a neighborhood of $r$ contained in $\{s:\sigma>r-1\}$.
For each compact $K\subset U$, put $\sigma_K=\inf_{s\in K}\sigma>r-1$.
The lower bound in the introduction implies
\[
\left|\partial_{\boldsymbol{x}}^{\boldsymbol{\mu}}f_s(\boldsymbol{x})\right|
\leq A_{K,\boldsymbol{\mu}}
(1+x_1+\cdots+x_r)^{-\sigma_K-\nu},
\qquad s\in K.
\]
The constants are independent of $M$ and can also be chosen uniformly
when the parameters range over a compact subset of $\mathcal{H}_{\phi}^{r+1}$.
The difference $D_{r,M}(s)$ is then expressed as a finite sum of
boundary integrals and remainder integrals.

Every boundary integral has dimension at most $r-1$ and is dominated,
uniformly for $s\in K$, by a constant times
\[
\int_0^\infty \cdots \int_0^\infty
(1+x_1+\cdots+x_d)^{-\sigma_K}
\,dx_1\cdots dx_d,
\qquad d\leq r-1,
\]
which is finite because $\sigma_K>d$.

Each remainder term contains at least one derivative with respect
to one of the variables and a bounded periodic Bernoulli factor.
Hence it is dominated, uniformly for $s\in K$, by a constant times
\[
\int_0^\infty \cdots \int_0^\infty
(1+x_1+\cdots+x_d)^{-\sigma_K-\nu}
\,dx_1\cdots dx_d,
\]
where $d\leq r$ and $\nu\geq1$.
These integrals are finite because $\sigma_K+\nu>d$.
Any term with a coordinate fixed at the upper endpoint $M$ is bounded by
a constant times $(1+M)^{d-\sigma_K-\nu}$ (with $\nu=0$ allowed
for boundary terms), and therefore tends to zero uniformly on $K$.
The other terms converge uniformly on $K$ by the same integrable bounds.
The triangle inequality also gives
$|a+\boldsymbol{x}\cdot\boldsymbol{w}|\leq C'(1+x_1+\cdots+x_r)$.
Together with the lower bound and the bounded imaginary part of the
fixed logarithm, this controls the logarithmic factors. Differentiating
with respect to $s$ gives finite sums containing powers of logarithms,
bounded by constants times $(1+\log(1+x_1+\cdots+x_r))^k$ for each fixed
$k\geq0$; these factors do not affect convergence under the strict inequalities above.

It follows that $D_{r,M}(s)$ converges locally uniformly on $U$ to
a function $H_r(s,a;\boldsymbol{w})$ holomorphic there.

For $s\in U$ with $\sigma>r$, both the finite sum and the finite integral converge
to their corresponding infinite expressions as $M\to\infty$.
Therefore,
\[
H_r(s,a;\boldsymbol{w})
=
\zeta_r(s,a;\boldsymbol{w})
-
J_r(s,a),
\]
which proves the assertion.
\end{proof}

\bigskip

\medskip

\begin{proof}[Proof of Theorem \ref{th:Main_Theorem1}]
The common half-plane assumption is in force throughout this proof.
Put
$
s=r+\varepsilon.
$
For the finite $r$-fold sum, we have
\begin{align*}
(a+\boldsymbol{m}\cdot\boldsymbol{w})^{-r-\varepsilon}
&=
\frac{1}{(a+\boldsymbol{m}\cdot\boldsymbol{w})^r}
\exp
\left(
-\varepsilon
\log(a+\boldsymbol{m}\cdot\boldsymbol{w})
\right)
\\
&=
\frac{1}{(a+\boldsymbol{m}\cdot\boldsymbol{w})^r}
\sum_{k=0}^{\infty}
\frac{(-1)^k}{k!}
\log^k(a+\boldsymbol{m}\cdot\boldsymbol{w})
\varepsilon^k.
\end{align*}
Hence the coefficient of $\varepsilon^k$ in
\[
\sum_{m_1=0}^M\cdots\sum_{m_r=0}^M
(a+\boldsymbol{m}\cdot\boldsymbol{w})^{-r-\varepsilon}
= 
\sum_{k=0}^{\infty}
\left\{
\frac{(-1)^k}{k!}
\sum_{m_1=0}^M\cdots\sum_{m_r=0}^M
\frac{
\log^k(a+\boldsymbol{m}\cdot\boldsymbol{w})
}{
(a+\boldsymbol{m}\cdot\boldsymbol{w})^r
}
\right\}
\varepsilon^{k}
\]
is
\begin{align}
\frac{(-1)^k}{k!}
\sum_{m_1=0}^M\cdots\sum_{m_r=0}^M
\frac{
\log^k(a+\boldsymbol{m}\cdot\boldsymbol{w})
}{
(a+\boldsymbol{m}\cdot\boldsymbol{w})^r
}.
\label{eq:finite-sum-coefficient}
\end{align}
Next, by Lemma \ref{lem:multiple_integral},
\begin{align*}
J_{r,M}(r+\varepsilon,a)
&=
\frac{1}{w_1\cdots w_r}
\frac{1}{\varepsilon}
\frac{1}
{(1+\varepsilon)(2+\varepsilon)\cdots(r-1+\varepsilon)}
\\
&\quad\times
\sum_{I\subseteq\{1,\ldots,r\}}
(-1)^{|I|}
\left(
a+\sum_{i\in I}w_iM
\right)^{-\varepsilon}.
\end{align*}
By Definition \ref{def:cqr},
\[
\frac{1}
{(1+\varepsilon)(2+\varepsilon)\cdots(r-1+\varepsilon)}
=
\sum_{q=0}^{\infty}
c_q^{(r)}\varepsilon^q,
\]
whereas
\[
\left(
a+\sum_{i\in I}w_iM
\right)^{-\varepsilon}
=
\sum_{n=0}^{\infty}
\frac{(-1)^n}{n!}
\log^n
\left(
a+\sum_{i\in I}w_iM
\right)
\varepsilon^n.
\]
Since
\[
\sum_{I\subseteq\{1,\ldots,r\}}(-1)^{|I|}
=
(1-1)^r
=
0,
\]
the terms corresponding to $n=0$ cancel.
It follows that the coefficient of $\varepsilon^k$ in
$J_{r,M}(r+\varepsilon,a)$ is
\begin{align}
&\frac{1}{w_1\cdots w_r}
\sum_{q=0}^{k}
c_q^{(r)}
\frac{(-1)^{k+1-q}}{(k+1-q)!}
\sum_{I\subseteq\{1,\ldots,r\}}
(-1)^{|I|}
\log^{k+1-q}
\left(
a+\sum_{i\in I}w_iM
\right).
\label{eq:finite-integral-coefficient}
\end{align}

By Lemma \ref{lem:regularized_difference},
\[
D_{r,M}(s)
=
\sum_{m_1=0}^M\cdots\sum_{m_r=0}^M
(a+\boldsymbol{m}\cdot\boldsymbol{w})^{-s}
-
J_{r,M}(s,a)
\]
converges locally uniformly in a neighborhood of $s=r$ to the
holomorphic function $H_r(s,a;\boldsymbol{w})$.
Cauchy's integral formula on a sufficiently small circle centered at
$r$ permits passage to the limit in every Taylor coefficient.
Therefore, comparing the coefficients of $\varepsilon^k$ in
\eqref{eq:finite-sum-coefficient} and
\eqref{eq:finite-integral-coefficient}, we obtain
\begin{align}
\frac{1}{k!}
H_r^{(k)}(r,a;\boldsymbol{w})
&=
\lim_{M\to\infty}
\Bigg\{
\frac{(-1)^k}{k!}
\sum_{m_1=0}^M\cdots\sum_{m_r=0}^M
\frac{
\log^k(a+\boldsymbol{m}\cdot\boldsymbol{w})
}{
(a+\boldsymbol{m}\cdot\boldsymbol{w})^r
}
\nonumber\\
&\qquad
-
\frac{1}{w_1\cdots w_r}
\sum_{q=0}^{k}
c_q^{(r)}
\frac{(-1)^{k+1-q}}{(k+1-q)!}
\sum_{I\subseteq\{1,\ldots,r\}}
(-1)^{|I|}
\nonumber\\
&\hspace{47mm}\times
\log^{k+1-q}
\left(
a+\sum_{i\in I}w_iM
\right)
\Bigg\}.
\label{eq:H-coefficient}
\end{align}
We now calculate the contribution of $J_r(s,a)$.
By Lemma \ref{lem:multiple_integral},
\begin{align*}
J_r(r+\varepsilon,a)
&=
\frac{a^{-\varepsilon}}
{w_1\cdots w_r\,
\varepsilon
(1+\varepsilon)\cdots(r-1+\varepsilon)}
\\
&=
\frac{1}{w_1\cdots w_r}
\frac{1}{\varepsilon}
\left(
\sum_{q=0}^{\infty}
c_q^{(r)}\varepsilon^q
\right)
\left(
\sum_{n=0}^{\infty}
\frac{(-1)^n\log^n a}{n!}
\varepsilon^n
\right).
\end{align*}
Thus
\[
J_r(r+\varepsilon,a)
=
\frac{c_0^{(r)}}{w_1\cdots w_r}
\frac{1}{\varepsilon}
+
\sum_{k=0}^{\infty}
\frac{1}{w_1\cdots w_r}
\left\{
\sum_{q=0}^{k+1}
c_q^{(r)}
\frac{(-1)^{k+1-q}\log^{k+1-q}a}
{(k+1-q)!}
\right\}
\varepsilon^k.
\]
Since
\[
c_0^{(r)}
=
\frac{1}{(r-1)!},
\]
the residue at $s=r$ is
\[
\frac{1}{(r-1)!w_1\cdots w_r}.
\]
On the other hand, by Lemma \ref{lem:regularized_difference},
\[
\zeta_r(s,a;\boldsymbol{w})
=
H_r(s,a;\boldsymbol{w})
+
J_r(s,a)
\]
in a neighborhood of $s=r$.
Hence
\begin{align*}
\gamma_k(r,a;\boldsymbol{w})
&=
\frac{1}{k!}
H_r^{(k)}(r,a;\boldsymbol{w})
+
\frac{1}{w_1\cdots w_r}
\sum_{q=0}^{k+1}
c_q^{(r)}
\frac{(-1)^{k+1-q}\log^{k+1-q}a}
{(k+1-q)!}.
\end{align*}
In \eqref{eq:H-coefficient}, the contribution from the subset
$I=\emptyset$ is
\[
-\frac{1}{w_1\cdots w_r}
\sum_{q=0}^{k}
c_q^{(r)}
\frac{(-1)^{k+1-q}\log^{k+1-q}a}
{(k+1-q)!}.
\]
This cancels exactly with the terms $q=0,\ldots,k$ arising from
$J_r(r+\varepsilon,a)$.
The only remaining term from $J_r$ is the term $q=k+1$, namely
\[
\frac{c_{k+1}^{(r)}}{w_1\cdots w_r}.
\]
Since
\[
\sum_{\emptyset\neq I\subseteq\{1,\ldots,r\}}(-1)^{|I|}=-1,
\]
this constant is exactly the contribution obtained by extending the
sum over $q$ in the nonempty-subset terms to $q=k+1$.
Therefore,
\begin{align*}
\gamma_k(r,a;\boldsymbol{w})
&=
\lim_{M\to\infty}
\Bigg\{
\frac{(-1)^k}{k!}
\sum_{m_1=0}^M\cdots\sum_{m_r=0}^M
\frac{
\log^k(a+\boldsymbol{m}\cdot\boldsymbol{w})
}{
(a+\boldsymbol{m}\cdot\boldsymbol{w})^r
}
\\
&\qquad
-
\frac{1}{w_1\cdots w_r}
\sum_{q=0}^{k+1}
c_q^{(r)}
\frac{(-1)^{k+1-q}}{(k+1-q)!}
\sum_{\emptyset\neq I\subseteq\{1,\ldots,r\}}
(-1)^{|I|}
\\
&\hspace{48mm}\times
\log^{k+1-q}
\left(
a+\sum_{i\in I}w_iM
\right)
\Bigg\},
\end{align*}
which proves \eqref{eq:gamma-r-main}.
Finally, differentiating the generating function in
Definition \ref{def:cqr} at $\varepsilon=0$, we obtain
\[
c_1^{(r)}
=
-\frac{1}{(r-1)!}
\sum_{\nu=1}^{r-1}\frac1{\nu}
=
-\frac{H_{r-1}}{(r-1)!}.
\]
Putting $k=0$ in \eqref{eq:gamma-r-main} therefore gives
\eqref{eq:gamma0-r-main}.
\end{proof}

\medskip

\begin{proof}[Proof of Theorem \ref{th:Main_Theorem2}]
We first recall the relation between differentiation with respect to
$a$ and the shift of the variable $s$.
Fix $\boldsymbol{w}$ and let $a$ vary in $\mathcal{H}_{\phi}$.
On each compact subset of this half-plane the lower bound in the introduction
holds with a uniform positive constant. The defining series and its
$a$-derivatives therefore converge locally uniformly for $\sigma>r$.
Here and below, differentiation in $a$ is holomorphic differentiation.
Thus, for $\sigma>r$, termwise differentiation gives
\[
\frac{\partial}{\partial a}
\zeta_r(s,a;\boldsymbol{w})
=
-s\zeta_r(s+1,a;\boldsymbol{w}).
\]
Repeating this differentiation $j$ times, we obtain
\begin{align}
\frac{\partial^j}{\partial a^j}
\zeta_r(s,a;\boldsymbol{w})
&=
(-1)^j
s(s+1)\cdots(s+j-1)
\zeta_r(s+j,a;\boldsymbol{w})
\nonumber\\
&=
(-1)^j
(s)_j
\zeta_r(s+j,a;\boldsymbol{w}).
\label{eq:a-derivative-shift}
\end{align}
To justify the holomorphic dependence on $a$ after continuation, fix
$a_0\in\mathcal{H}_{\phi}$ and put
$\alpha_0=\Re(e^{-i\phi}a_0)>0$.
The disk $|h|<\alpha_0$ satisfies $a_0+h\in\mathcal H_\phi$.
The fixed holomorphic logarithm makes the local binomial expansion
consistent at every summand. For $\sigma>r$, it gives
\[
\zeta_r(s,a_0+h;\boldsymbol{w})
=\sum_{n=0}^{\infty}\frac{(-1)^n(s)_n}{n!}
\zeta_r(s+n,a_0;\boldsymbol{w})h^n,
\qquad (s)_0=1.
\]
For $s$ in a compact set $K$, $(s)_n/n!$ grows at most polynomially
in $n$, uniformly on $K$. Choose an integer $N$ such that
$\inf_{s\in K}\sigma+N>r$, and put
$A_{\boldsymbol m}=a_0+\boldsymbol m\cdot\boldsymbol w$.
For $n\geq N$, the bound $|A_{\boldsymbol m}|\geq\alpha_0$ gives
\[
\sum_{\boldsymbol m\in\mathbb Z_{\geq0}^r}
|A_{\boldsymbol m}^{-s-n}|
\leq\alpha_0^{N-n}
\sum_{\boldsymbol m\in\mathbb Z_{\geq0}^r}
|A_{\boldsymbol m}^{-s-N}|
\leq C_K\alpha_0^{-n},\qquad s\in K.
\]
The last sum is uniformly bounded by absolute convergence in the
shifted half-plane. Hence
$|\zeta_r(s+n,a_0;\boldsymbol w)|\leq C_K\alpha_0^{-n}$ on $K$.
Thus the tail converges normally for $|h|\leq\rho<\alpha_0$.
The finitely many remaining terms are meromorphic in $s$ and holomorphic
in $h$. At $s=-\ell$ with $\ell\geq0$, a pole of
$\zeta_r(s+n,a_0;\boldsymbol{w})$ requires $n\geq\ell+1$ and is
cancelled by the zero of $(s)_n$.
In particular, all terms are holomorphic near $s=0$.
This proves joint holomorphy near $s=0$ and local holomorphic dependence
on $a$ away from the possible poles.
Consequently, \eqref{eq:a-derivative-shift} holds meromorphically in $s$
by analytic continuation, with holomorphic dependence on $a$.
Since $\zeta_r(s,a;\boldsymbol{w})$ is holomorphic at $s=0$,
its Taylor expansion at $s=0$ is
\[
\zeta_r(s,a;\boldsymbol{w})
=
\sum_{m=0}^{\infty}
\frac{1}{m!}
\zeta_r^{(m)}(0,a;\boldsymbol{w})s^m.
\]
Joint holomorphy near $s=0$ and local uniform convergence of this Taylor
series justify differentiating with respect to $a$ $j$ times, yielding
\begin{align}
\frac{\partial^j}{\partial a^j}
\zeta_r(s,a;\boldsymbol{w})
=
\sum_{m=0}^{\infty}
\frac{1}{m!}
\frac{\partial^j}{\partial a^j}
\zeta_r^{(m)}(0,a;\boldsymbol{w})s^m.
\label{eq:Taylor-a-derivative}
\end{align}
On the other hand, from \eqref{eq:a-derivative-shift},
\[
\zeta_r(s+j,a;\boldsymbol{w})
=
(-1)^j
\frac{1}{(s)_j}
\frac{\partial^j}{\partial a^j}
\zeta_r(s,a;\boldsymbol{w}).
\]
By the definition of $d_{j,n}$,
\[
\frac{1}{(s)_j}
=
\sum_{n=-1}^{\infty}d_{j,n}s^n.
\]
Substituting this expansion and \eqref{eq:Taylor-a-derivative}
into the preceding identity, we obtain
\begin{align}
\zeta_r(s+j,a;\boldsymbol{w})
&=
(-1)^j
\left(
\sum_{n=-1}^{\infty}d_{j,n}s^n
\right)
\left(
\sum_{m=0}^{\infty}
\frac{1}{m!}
\frac{\partial^j}{\partial a^j}
\zeta_r^{(m)}(0,a;\boldsymbol{w})s^m
\right)
\nonumber\\
&=
(-1)^j
\sum_{k=-1}^{\infty}
\left\{
\sum_{m=0}^{k+1}
d_{j,k-m}
\frac{1}{m!}
\frac{\partial^j}{\partial a^j}
\zeta_r^{(m)}(0,a;\boldsymbol{w})
\right\}
s^k.
\label{eq:product-expansion}
\end{align}
The Laurent expansion of $\zeta_r(s,a;\boldsymbol{w})$ at $s=j$
is
\[
\zeta_r(s,a;\boldsymbol{w})
=
\frac{\gamma_{-1}(j,a;\boldsymbol{w})}{s-j}
+
\sum_{k=0}^{\infty}
\gamma_k(j,a;\boldsymbol{w})(s-j)^k.
\]
Replacing $s$ by $s+j$, we have
\[
\zeta_r(s+j,a;\boldsymbol{w})
=
\sum_{k=-1}^{\infty}
\gamma_k(j,a;\boldsymbol{w})s^k.
\]
Comparing the coefficient of $s^k$ with
\eqref{eq:product-expansion}, we conclude that
\[
\gamma_k(j,a;\boldsymbol{w})
=
(-1)^j
\sum_{m=0}^{k+1}
d_{j,k-m}
\frac{1}{m!}
\frac{\partial^j}{\partial a^j}
\zeta_r^{(m)}(0,a;\boldsymbol{w}),
\]
for $k=-1,0,1,2,\ldots$.
This proves the theorem.
\end{proof}

\medskip

\section{Asymptotic behavior of the Laurent coefficients}
\label{sec:large-order}

Throughout this section, $r\geq2$ and the parameters
$a,w_1,\ldots,w_r\in\mathcal H_\phi$ are fixed, with the logarithm
determination specified in the introduction.
All limits are taken as the coefficient order tends to infinity,
not as a parameter tends to infinity.
Theorem~\ref{th:all-poles-coefficients} gives the basic coefficient
formula for this analysis. Unlike the first two main results, it uses
only the meromorphic continuation and the residues: removing all
principal parts reduces the problem to Cauchy estimates for an entire
function.

For $1\leq q\leq r$, write
\begin{equation}
R_q=R_q(a;\boldsymbol w)
:=\Res_{s=q}\zeta_r(s,a;\boldsymbol w)
=\frac{(-1)^{r-q}}{(q-1)!(r-q)!}
B_{r-q}^{(r)}(a;\boldsymbol w),
\label{eq:large-order-residues}
\end{equation}
where the last equality is \eqref{residue_intro}.
Define, initially away from $s=1,\ldots,r$,
\begin{equation}
E_r(s)=E_r(s,a;\boldsymbol w)
:=\zeta_r(s,a;\boldsymbol w)-\sum_{q=1}^r\frac{R_q}{s-q}.
\label{eq:entire-part}
\end{equation}
Every possible pole is simple and its principal part has been
subtracted, so all the remaining singularities are removable.
We use $E_r$ for the resulting entire function, to distinguish it
from $H_r=\zeta_r-J_r$ in Lemma~\ref{lem:regularized_difference},
which is not the same regularization.
This reasoning applies without change to the complex parameters
under consideration. If $R_j=0$, the coefficients at $s=j$ are
ordinary Taylor coefficients, and the statements below still apply.

\begin{theorem}\label{th:all-poles-coefficients}
For $j\in\{1,\ldots,r\}$ and every integer $k\geq0$,
\begin{equation}
\gamma_k(j,a;\boldsymbol w)
=(-1)^k\sum_{q=1}^{j-1}\frac{R_q}{(j-q)^{k+1}}
-\sum_{q=j+1}^r\frac{R_q}{(q-j)^{k+1}}
+\frac{E_r^{(k)}(j)}{k!}.
\label{eq:all-poles-coefficients}
\end{equation}
For every $R>0$, let
\[
M_{j,R}:=\max_{|s-j|=R}|E_r(s)|.
\]
Then
\begin{equation}
\left|
\gamma_k(j,a;\boldsymbol w)
-(-1)^k\sum_{q=1}^{j-1}\frac{R_q}{(j-q)^{k+1}}
+\sum_{q=j+1}^r\frac{R_q}{(q-j)^{k+1}}
\right|
\leq \frac{M_{j,R}}{R^k}.
\label{eq:all-poles-cauchy-bound}
\end{equation}
Empty sums are understood to be zero.
\end{theorem}

\begin{proof}
Put $z=s-j$. For $q<j$, the geometric series gives
\[
\frac1{s-q}
=\frac1{j-q}\frac1{1+z/(j-q)}
=\sum_{k=0}^{\infty}\frac{(-1)^k z^k}{(j-q)^{k+1}},
\qquad |z|<j-q.
\]
For $q>j$, on the other hand,
\[
\frac1{s-q}
=-\frac1{q-j}\frac1{1-z/(q-j)}
=-\sum_{k=0}^{\infty}\frac{z^k}{(q-j)^{k+1}},
\qquad |z|<q-j.
\]
Thus the contributions from the right-hand poles have a minus sign
independent of $k$. Expand \eqref{eq:entire-part} near $s=j$,
retain $R_j/(s-j)$ as the principal part, and compare the regular
coefficients to obtain \eqref{eq:all-poles-coefficients}.
Since $E_r$ is entire, for any positively oriented circle of radius
$R>0$ centered at $j$ we have
\[
\frac{E_r^{(k)}(j)}{k!}
=\frac1{2\pi i}\int_{|s-j|=R}
\frac{E_r(s)}{(s-j)^{k+1}}\,ds.
\]
Bounding the integrand and using the circle length $2\pi R$
gives \eqref{eq:all-poles-cauchy-bound}.
\end{proof}

In particular, the entire-part coefficient is $O(R^{-k})$ for
every fixed $R>0$. Its constant may depend on $j$, $R$, the parameters,
and the logarithm determination, but not on $k$.
No uniformity as $R\to\infty$ is asserted, and $M_{j,R}$ is not
an explicit parameter-only bound.
The all-poles expansion must be distinguished from a truncation
to neighboring poles. Indeed, taking $R=2$ gives
\begin{equation}
\left|
\gamma_k(j,a;\boldsymbol w)
-(-1)^k\sum_{\substack{q<j\\j-q=1}}R_q
+\sum_{\substack{q>j\\q-j=1}}R_q
\right|
\leq 2^{-k}\left(M_{j,2}
+\sum_{\substack{1\leq q\leq r\\|q-j|\geq2}}
\frac{|R_q|}{|q-j|}\right).
\label{eq:neighbor-remainder}
\end{equation}
This follows from $|q-j|^{-k}\leq2^{-k}$ for the omitted poles.

\begin{corollary}\label{cor:highest-large-order}
At the highest pole, as $k\to\infty$,
\begin{align}
\gamma_k(r,a;\boldsymbol w)
&=(-1)^kR_{r-1}+O(2^{-k})\nonumber\\
&=(-1)^k\frac{w_1+\cdots+w_r-2a}
{2(r-2)!w_1\cdots w_r}+O(2^{-k}).
\label{eq:highest-large-order}
\end{align}
Consequently the following limit exists in $\mathbb C$:
\begin{equation}
\lim_{k\to\infty}(-1)^k\gamma_k(r,a;\boldsymbol w)
=\frac{w_1+\cdots+w_r-2a}
{2(r-2)!w_1\cdots w_r}.
\label{eq:highest-large-order-limit}
\end{equation}
\end{corollary}

\begin{proof}
At $j=r$ there is no pole to the right, and
\eqref{eq:all-poles-coefficients} reads
\[
\gamma_k(r,a;\boldsymbol w)
=(-1)^k\sum_{q=1}^{r-1}\frac{R_q}{(r-q)^{k+1}}
+\frac{E_r^{(k)}(r)}{k!}.
\]
Equation \eqref{eq:neighbor-remainder} isolates $q=r-1$.
To evaluate its residue, expand \eqref{Bernoulli_intro} to first order:
\[
\frac{t^r e^{at}}{\prod_{i=1}^r(e^{w_i t}-1)}
=\frac1{w_1\cdots w_r}
\left(1+\left(a-\frac12\sum_{i=1}^r w_i\right)t+O(t^2)\right).
\]
Hence
\[
B_1^{(r)}(a;\boldsymbol w)
=\frac{2a-w_1-\cdots-w_r}{2w_1\cdots w_r},
\qquad
R_{r-1}=-\frac{B_1^{(r)}(a;\boldsymbol w)}{(r-2)!}.
\]
Substitution proves both assertions, since the error tends to zero
in complex modulus.
\end{proof}

\begin{corollary}\label{cor:lowest-large-order}
At $s=1$ we have
\begin{equation}
\gamma_k(1,a;\boldsymbol w)=-R_2+O(2^{-k}),
\qquad
\lim_{k\to\infty}\gamma_k(1,a;\boldsymbol w)=-R_2.
\label{eq:lowest-large-order}
\end{equation}
The limit is taken in $\mathbb C$.
\end{corollary}

\begin{proof}
There is no pole to the left, so
\[
\gamma_k(1,a;\boldsymbol w)
=-\sum_{q=2}^r\frac{R_q}{(q-1)^{k+1}}
+\frac{E_r^{(k)}(1)}{k!}.
\]
Isolating $q=2$ and applying \eqref{eq:neighbor-remainder} proves
\eqref{eq:lowest-large-order}.
\end{proof}

\begin{corollary}\label{cor:intermediate-large-order}
For $r\geq3$ and $2\leq j\leq r-1$,
\begin{equation}
\gamma_k(j,a;\boldsymbol w)
=(-1)^kR_{j-1}-R_{j+1}+O(2^{-k}).
\label{eq:intermediate-large-order}
\end{equation}
In particular,
\begin{align}
\lim_{n\to\infty}\gamma_{2n}(j,a;\boldsymbol w)
&=R_{j-1}-R_{j+1},\nonumber\\
\lim_{n\to\infty}\gamma_{2n+1}(j,a;\boldsymbol w)
&=-R_{j-1}-R_{j+1}.
\label{eq:even-odd-limits}
\end{align}
Both limits exist in $\mathbb C$. The full sequence converges if
and only if $R_{j-1}=0$, in which case its limit is $-R_{j+1}$.
\end{corollary}

\begin{proof}
Both neighboring indices are present in \eqref{eq:neighbor-remainder}.
Substituting $k=2n$ and $k=2n+1$ gives the two limits.
They coincide exactly when $2R_{j-1}=0$, which is also sufficient
for convergence by \eqref{eq:intermediate-large-order}.
\end{proof}

\begin{remark}\label{rem:vanishing-large-order}
The preceding formulae remain valid when a residue vanishes;
the corresponding point is then removable rather than a pole.
Thus ``nearest poles'' means the nearest poles with nonzero residues.
More generally, retaining in \eqref{eq:all-poles-coefficients}
all contributions with $1\leq|q-j|\leq D$, for a fixed integer
$D\geq1$, leaves an $O((D+1)^{-k})$ remainder: use Cauchy's
estimate at radius $D+1$ and bound every omitted pole by its distance.
If no nonzero pole contribution is omitted, the remainder is
$O(R^{-k})$ for every fixed $R>0$.
For example, if $r\geq3$ and $R_{r-1}=0$, then
\[
\gamma_k(r,a;\boldsymbol w)
=(-1)^k\frac{R_{r-2}}{2^{k+1}}+O(3^{-k}).
\]
This is a nonzero leading term only if $R_{r-2}\neq0$.
If all other residues vanish, only the entire-part coefficient remains.
At an intermediate point, equidistant contributions may also cancel
on one parity: their sum at distance $d$ is
$\bigl((-1)^kR_{j-d}-R_{j+d}\bigr)/d^{k+1}$ whenever both indices
are present. If this vanishes for even or odd $k$, the next nonzero
contribution on that subsequence must be considered.
\end{remark}

\begin{remark}\label{rem:double-large-order}
For $r=2$, the only residues are
\[
R_1=\frac{w_1+w_2-2a}{2w_1w_2},\qquad
R_2=\frac1{w_1w_2}.
\]
Thus, for every fixed $R>0$,
\begin{align*}
\gamma_k(2,a;w_1,w_2)
&=(-1)^k\frac{w_1+w_2-2a}{2w_1w_2}+O(R^{-k}),\\
\gamma_k(1,a;w_1,w_2)
&=-\frac1{w_1w_2}+O(R^{-k}).
\end{align*}
Taking any $R>1$ recovers the two limits, in particular
$\gamma_k(1,a;w_1,w_2)\to-1/(w_1w_2)$.
These are the asymptotic forms considered in
\cite{Miyagawa2025_LEBDZ}*{Theorems~2.6 and~2.7}, with
$(a,w_1,w_2)=(\alpha,v,w)$ in that paper.
Here the remainders use $M_{j,R}$, rather than a particular Mellin
integral bound, and allow arbitrary fixed radii.
There are no omitted poles in this double-zeta case, which explains
why the stronger remainder holds after retaining just one pole.
\end{remark}

\begin{remark}\label{rem:triple-large-order}
For $r=3$ and $\boldsymbol w=(1,1,1)$, the generating function gives
\[
R_1=\frac{a^2-3a+2}{2},\qquad
R_2=\frac32-a,\qquad R_3=\frac12.
\]
For $a>0$, this can also be checked independently by counting triples
with a given sum:
\[
\zeta_3(s,a;1,1,1)
=\frac12\zeta_H(s-2,a)
+\left(\frac32-a\right)\zeta_H(s-1,a)
+\frac{a^2-3a+2}{2}\zeta_H(s,a),
\]
using compatible logarithms. Consequently, for every fixed $R>0$,
\begin{align*}
\gamma_k(3,a;1,1,1)
&=(-1)^k\left(R_2+\frac{R_1}{2^{k+1}}\right)+O(R^{-k}),\\
\gamma_k(1,a;1,1,1)
&=-R_2-\frac{R_3}{2^{k+1}}+O(R^{-k}),\\
\gamma_k(2,a;1,1,1)
&=(-1)^kR_1-R_3+O(R^{-k}).
\end{align*}
For $a=3/2$, $R_2=0$ but $R_1=-1/8$, so the highest-pole
coefficient has the nonzero contribution
$(-1)^{k+1}/2^{k+4}$ from distance two.
For $a=1$, $R_1=0$, and the even and odd limits at $j=2$
both equal $-1/2$.
These cases illustrate why a zero nearest residue must not be
interpreted as the absence of every more distant contribution.
\end{remark}

\section*{Acknowledgments}
At the 2026 Autumn Meeting of the Mathematical Society of Japan, held at Kobe University, I received valuable comments from Prof. Genki Shibukawa, which contributed to the further development of this paper. I would like to take this opportunity to express my deep gratitude.
I would also like to thank the members of my laboratory and the Kansai Multiple Zeta Study Group.

\bibliographystyle{amsalpha}
\bibliography{References} 

@Article{Barnes1899,
 Author = {Barnes, E. W.},
 Title = {The genesis of the double gamma functions.},
 FJournal = {Proceedings of the London Mathematical Society},
 Journal = {Proc. Lond. Math. Soc.},
 ISSN = {0024-6115},
 Volume = {31},
 Pages = {358--381},
 Year = {1899},
 Language = {English},
 DOI = {10.1112/plms/s1-31.1.358},
 URL = {zenodo.org/record/1447742},
 zbMATH = {2668383},
 JFM = {30.0389.03}
}

@Article{Barnes1901,
 Author = {Barnes, E. W.},
 Title = {The theory of the double gamma function.},
 FJournal = {Philosophical Transactions of the Royal Society of London, Series A},
 Journal = {Philos. Trans. R. Soc. Lond., Ser. A, Contain. Pap. Math. Phys. Character},
 Volume = {196},
 Pages = {265--387},
 Year = {1901},
 Language = {English},
 DOI = {10.1098/rsta.1901.0006},
 zbMATH = {2662859},
 JFM = {32.0442.02}
}

@Misc{Barnes1904,
 Author = {Barnes, E. W.},
 Title = {On the theory of the multiple {Gamma} function.},
 Year = {1904},
 Language = {English},
 HowPublished = {Cambr. {Trans}. 19, 374-425 (1904).},
 zbMATH = {2653713},
 JFM = {35.0462.01}
}

@article{Berndt1972,
 author = {Berndt, Bruce C.},
 title = {On the {Hurwitz} zeta-function},
 fjournal = {Rocky Mountain Journal of Mathematics},
 journal = {Rocky Mt. J. Math.},
 issn = {0035-7596},
 volume = {2},
 pages = {151--157},
 year = {1972},
 language = {English},
 doi = {10.1216/RMJ-1972-2-1-151},
 zbMATH = {3363666},
 Zbl = {0229.10023}
}

@book{Finch2003,
 author = {Finch, Steven R.},
 title = {Mathematical constants},
 fseries = {Encyclopedia of Mathematics and Its Applications},
 series = {Encycl. Math. Appl.},
 issn = {0953-4806},
 volume = {94},
 isbn = {0-521-81805-2},
 year = {2003},
 publisher = {Cambridge: Cambridge University Press},
 language = {English},
 zbMATH = {2018401},
 Zbl = {1054.00001}
}

@article{MatsumotoOnozukaWakabayashi2020,
 author = {Matsumoto, Kohji and Onozuka, Tomokazu and Wakabayashi, Isao},
 title = {Laurent series expansions of multiple zeta-functions of {Euler}-{Zagier} type at integer points},
 fjournal = {Mathematische Zeitschrift},
 journal = {Math. Z.},
 issn = {0025-5874},
 volume = {295},
 number = {1-2},
 pages = {623--642},
 year = {2020},
 language = {English},
 doi = {10.1007/s00209-019-02337-2},
 zbMATH = {7203131},
 Zbl = {1445.11089}
}

@misc{Miyagawa2025_LEBDZ,
 author = {Takashi Miyagawa},
 title = {On the {Laurent} series expansions of the {Barnes} double zeta function},
 year = {2026},
 howpublished = {Preprint, {arXiv}:2507.18152 [math.{NT}] (2026)},
 url = {https://arxiv.org/abs/2507.18152},
 arXiv = {arXiv:2507.18152}
}

@article{Noronha2017,
 author = {Noronha, J. M. B.},
 title = {Representations for the derivative at zero and finite parts of the {Barnes} zeta function},
 journal = {Integral Transforms Spec. Funct.},
 volume = {28},
 number = {6},
 pages = {423--442},
 year = {2017},
 doi = {10.1080/10652469.2017.1304937}
}

@article{Ruijsenaars2000,
 author = {Ruijsenaars, S. N. M.},
 title = {On {Barnes}' multiple zeta and gamma functions},
 fjournal = {Advances in Mathematics},
 journal = {Adv. Math.},
 issn = {0001-8708},
 volume = {156},
 number = {1},
 pages = {107--132},
 year = {2000},
 language = {English},
 doi = {10.1006/aima.2000.1946},
 zbMATH = {1579914},
 Zbl = {0966.33013}
}

@article{Spreafico2009,
 author = {Spreafico, M.},
 title = {On the {Barnes} double zeta and gamma functions},
 fjournal = {Journal of Number Theory},
 journal = {J. Number Theory},
 issn = {0022-314X},
 volume = {129},
 number = {9},
 pages = {2035--2063},
 year = {2009},
 language = {English},
 doi = {10.1016/j.jnt.2009.03.005},
 zbMATH = {5575276},
 Zbl = {1185.11055}
}

@article{ZhangWilliams1994,
 author = {Zhang, Nanyue and Williams, Kenneth S.},
 title = {Some results on the generalized {Stieltjes} constants},
 fjournal = {Analysis},
 journal = {Analysis},
 issn = {0174-4747},
 volume = {14},
 number = {2-3},
 pages = {147--162},
 year = {1994},
 language = {English},
 zbMATH = {687756},
 Zbl = {0808.11054}
}

\end{document}